\documentclass[11pt]{amsart}

\usepackage{amsmath, amssymb, amsthm, amsfonts, bbm, dsfont}
\usepackage{graphicx}

\usepackage{hyperref}

\usepackage{a4wide}

\numberwithin{equation}{section}
\theoremstyle{plain}
\newtheorem{theorem}[subsubsection]{Theorem}
 \newtheorem{lemma}[subsubsection]{Lemma}
 \newtheorem{prop}[subsubsection]{Proposition}
 
 \newtheorem{conj}[subsubsection]{Conjecture}

 \theoremstyle{definition}

\usepackage{hyperref}
\usepackage{pictexwd}
\usepackage{stackengine}
\usepackage{mathtools}
\usepackage{amssymb}
\usepackage{dsfont}

\usepackage{mathabx}

\newcommand{\CC}{\mathbb{C}}

\newcommand{\QQ}{\mathbb{Q}}
\newcommand{\ZZ}{\mathbb{Z}}

\newcommand{\calF}{\mathcal{F}}
\newcommand{\calG}{\mathcal{G}}
\newcommand{\calO}{\mathcal{O}}

\newcommand{\frb}{\mathfrak{b}}

\newcommand{\frn}{\mathfrak{n}}

\newcommand{\ind}{\textup{ind}}

\newcommand{\Tr}{\textup{Tr}}

\newcommand\Hom{\textup{Hom}}
\newcommand\End{\textup{End}}

\newcommand\GL{\textup{GL}}
\newcommand\gl{\mathfrak{gl}}

\newcommand{\ad}{\textup{ad}}
\newcommand{\Ad}{\textup{Ad}}

\newcommand{\quash}[1]{}

\newcommand{\und}{\underline}

\newcommand{\Br}{\mathfrak{Br}}

\usepackage{tikz}
\usetikzlibrary{shapes,fit,intersections}
\usetikzlibrary{decorations.markings}
\usetikzlibrary{decorations.pathreplacing}
\usetikzlibrary{patterns}
\usetikzlibrary{matrix,arrows}
\usepgflibrary{decorations.pathmorphing}

\usepackage{tikz-cd}

\newcommand{\Wr}{\bar{W}}

\newcommand{\MF}{\mathbf{MF}^{\hat{T}}}
\newcommand{\calXr}{\bar{\mathcal{X}}}

\newcommand{\MFs}{\mathbf{MF}}

\newcommand{\calC}{\mathcal{C}}

\newcommand{\JM}{\mathrm{JM}}
\newcommand{\CE}{\mathrm{CE}}

\newcommand{\calCr}{\bar{\mathcal{C}}}
\newcommand{\fgt}{\mathrm{fgt}}

\DeclareMathOperator{\xIm}{Im}

\thanks{The work of A.O. was supported in part by  the NSF CAREER grant DMS-1352398}
\thanks{The work of L.R. was supported in part by the Sloan Foundation and the NSF grant DMS-1108727}

\title{A categorification of a cyclotomic Hecke algebra}

\author{A. Oblomkov}
\address{
A.~Oblomkov\\
Department of Mathematics and Statistics\\
University of Massachusetts at Amherst\\
Lederle Graduate Research Tower\\
710 N. Pleasant Street\\
Amherst, MA 01003 USA
}
\email{oblomkov@math.umass.edu}

\author{L. Rozansky}
\address{
L.~Rozansky\\
Department of Mathematics\\
University of North Carolina at Chapel Hill\\
CB \# 3250, Phillips Hall\\
Chapel Hill, NC 27599 USA
}
\email{rozansky@math.unc.edu}

\begin{document}
\maketitle
\begin{abstract}
  We propose a categorification of the cyclotomic Hecke algebra in terms
  of the equivariant K-theory of the framed matrix factorizations. The
  construction generalizes the earlier construction of the authors
  for a categorification of the finite Hecke algebra of type A.
  We also explain why our construction provides a faithful realization of the
  Hecke algebras and discuss a geometric realization of the Jucys-Murphy subalgebra.
\end{abstract}

\tableofcontents
\section{Introduction}
\label{sec:introduction}

In this paper we expand the approach of \cite{OblomkovRozansky16} and \cite{OblomkovRozansky17} in order to categorify the cyclotomic Hecke algebra with the help of a special category of matrix factorizations.

\def\IC{ \mathbb{C}}
\def\GL{ \mathrm{GL}}
\def\GLn{\GL_n}
\def\GLnC{ \GL(n,\IC)}
\def\gl{ \mathrm{gl}}
\def\mG{ G }
\def\mB{ B }
\def\mfg{ \mathfrak{g}}
\def\mfh{ \mathfrak{h}}
\def\mfb{ \mathfrak{b}}
\def\mfn{ \mathfrak{n}}
\def\mfB{ \mathfrak{B}}
\def\csX{ \mathcal{X}}
\def\xact{\cdot}
\def\rmT{ \mathrm{T}}
\def\rmTs{ \rmT^* }
\def\TsB{ \rmTs\mfB }

\def\Adv#1{ \mathrm{Ad}_{#1}}

\def\mBB{\mB\times\mB}

\def\xCat{ \mathbf{MF}}
\def\xCatfrr{ \xCat^{\xfrm}_{\xr}}
\def\xCatfrrp{ \xCat^{\xfrm}_{\xrp}}

\def\xCatn{ \xCat_n}
\def\xCatnfrr{ \xCat^{\xfrm}_{n,\xr}}
\def\xCatnfrrp{ \xCat^{\xfrm}_{n,\xrp}}

\def\xCatnx{ \xCatn }
\def\xCatny{ \xCatn }

\def\xfrm{ \mathrm{fr}}
\def\xfrmd{ \mathrm{f}}
\def\xfr{\mathrm{fr}}
\def\xfree{ \mathrm{free}}
\def\cycl{\mathrm{ct}}
\def\xfull{ \mathrm{full}}
\def\xsf{ \mathrm{sf}}
\def\xaff{\mathrm{aff}}
\def\xcr{ \mathrm{crit}}

\def\xMF{ \mathbf{MF}}
\def\xMFT{ \xMF^{\Tqt}}
\def\xMFTv#1{ \xMFT_{#1}}
\def\xMFTG{ \xMFTv{\mG}}
\def\xMFTB{ \xMFTv{\mB}}
\def\xMFTBB{ \xMFTv{\mBB}}

\def\xMFTa{ \xMF^{\Tqta}}
\def\xMFTav#1{ \xMFTa_{#1}}
\def\xMFTaG{ \xMFTav{\mG}}
\def\xMFTaB{ \xMFTav{\mB}}
\def\xMFTaBB{ \xMFTav{\mBB}}

\def\Tqt{ T_{\mathrm{qt}} }
\def\Tba{ T_{\xba}}
\def\Tqta{ \hat{T}}

\def\Cs{ \IC^* }
\def\Csq{ \Cs_q }
\def\sCsq{ \Csq }
\def\Cst{ \Cs_t }
\def\Csa{ \Cs_{\xba}}

\def\Knrr{Kn\"{o}rrer}
\def\Knrrp{\Knrr\ periodicity}

\def\yW{ \bar{W} }
\def\xW {\bar{ W} }
\def\xWl{ \xW^{\mathrm{lin}}}
\def\cXt{\bar{ \mathcal{X}}}
\def\cXtt{ \cXt_2 }
\def\cXth{ \cXt_3 }
\def\cXtd{ \mfb\times\mG\times\mfn }
\def\cXtfl{ \cXt_{\xfull}}
\def\cXtsf{ \cXt^{\xfrm}_r}

\def\cXttfl{ \cXt_{2,\xfull}}
\def\cXttsf{ \cXt^{\xfrm}_{2,r}}
\def\cXttsfB{ \cXttsf(B)}
\def\cXttscr{ \cXt^{\xcr}_{2,r}}
\def\cXttsfcr{ \cXt^{\xfrm,\xcr}_{2,r}}
\def\cXtts{ \cXt_{2,r}}

\def\xIcr{ I^{\xcr}}

\def\mfBr{ \mathfrak{Br}}
\def\mfBrv#1{ \mfBr_{#1}}
\def\mfBrn{ \mfBrv{n}}

\def\mfBraff{\mfBr^{\xaff}}
\def\mfBraffv#1{ \mfBraff_{#1}}
\def\mfBraffn{ \mfBraffv{n}}

\def\hPhi{ \Phi }
\def\hphi{ \phi }
\def\xPaff{ \hPhi^{\xaff}}
\def\xPct{ \hPhi^{\cycl}}
\def\xpaff{ \hphi^{\xaff}}

\def\xr{ r }
\def\xrp{ \xr' }

\def\mfH{ \mathfrak{H}}
\def\mfHv#1{ \mfH_{#1}}
\def\mfHn{ \mfHv{n}}

\def\mfHaff{ \mfH^{\xaff}}
\def\mfHaffv#1{ \mfHaff_{#1}}
\def\mfHaffn{ \mfHaffv{n}}

\def\mfHcl{ \mfH^{\cycl}}
\def\mfHr{ \mfHcl_{\xr}}
\def\mfHrp{ \mfHcl_{\xrp}}
\def\mfHo{ \mfHcl_{1}}
\def\mfHnr{ \mfHcl_{n,\xr}}
\def\mfHnrp{ \mfHcl_{n,\xrp}}
\def\mfHno{ \mfHcl_{n,1}}

\def\mfJM{ \mathfrak{J}}
\def\mfJM{\JM}
\def\mfJMcl{ \mfJM^{\cycl}}
\def\mfJMr{ \mfJMcl_{\xr}}
\def\mfJMnr{ \mfJMcl_{n,\xr}}

\def\xJM{Jucys-Murphy}
\def\xBL{Bernstein-Lusztig}
\def\yJM{\mathrm{JM}}
\def\yJMv#1{ \yJM_{#1}}
\def\yJMn{ \yJMv{n}}

\def\xKth{ \mathrm{K}}
\def\xKthv#1{ \xKth_{#1}}
\def\xKthq{ \xKthv{\Csq\times\Csa}}
\def\xKthqa{ \xKthv{\Csq\times\Csa}}

\def\xqtm{ \kappa }

\def\xcV{ \mathcal{V} }
\def\xcVv#1{\xcV_{#1}}
\def\xcVr{ \xcVv{\xr}}

\def\xv{ v }

\def\xstab{ \mathrm{stab}}

\def\ffr{\fgt^*}  
\def\ffrrrp{ (\fgt_{\xr}^{\xrp})^*}

\def\xba{ \mathbf{a}}
\def\xbrma{ \mathrm{\xba}}
\def\xbap {\xba' }

\def\dIC{ \IC^*_{\Delta}}
\def\cIC{ \IC^*_{\mathrm{c}}}

\def\sg{ \sigma }

\def\xa{ a }
\def\bxa{ \mathbf{\xa}}
\def\qpmo{ q^{\pm 1}}

\def\Dlt{\Delta}
\def\Dlto{\Dlt_n}

\def\finj{ \psi_{\mathrm{inj}}}
\def\fsur{\mathfrak{fgt}}
\def\xphifr{\mathrm{fgt}^*}
\def\fchifr{\mathfrak{fgt}}
\def\fchifrv#1{ \fchifr^{#1}}
\def\fchifrr{ \fchifrv{\xr}}
\def\fchifrvv#1#2{ \fchifr_{#1}^{#2}}
\def\fchifrrrp{ \fchifrvv{\xr}{\xrp}}
\def\fct{ \kappa^{\cycl}}

\def\xphict{ \phi^{\cycl}}

\def\Zqoa{ \ZZ[\qpmo,\bxa] }
\def\Qqoa{ \QQ[\qpmo,\bxa]}

\def\FHilb{ \mathrm{FHilb}}
\def\FHilbv#1{ \FHilb_{#1}}
\def\FHilbn{ \FHilbv{n}}
\def\FHilbno{ \FHilbv{n-1}}
\def\FHilbnC{ \bigl(\FHilbn \bigr)^{\Csq}}

\def\FHilbfr{ \FHilb^{\xfree}}
\def\FHilbfrv#1{ \FHilbfr_{#1}}
\def\FHilbfrn{ \FHilbfrv{n}}
\def\FHilbfrno{ \FHilbfrv{n-1}}

\def\FM{ \mathrm{FM}}
\def\FMvv#1#2{ \FM_{#1,#2}}
\def\FMrn{ \FMvv{r}{n}}
\def\FMon{ \FMvv{1}{n}}
\def\FMrnT{ \FMrn^{\Tqta}}

\def\tFM{ \widetilde{\FM}}
\def\tFMvv#1#2{ \tFM_{#1,#2}}
\def\tFMrn{ \tFMvv{r}{n}}

\def\FMfr{ \FM^{\xfree}}
\def\FMfrvv#1#2{ \FMfr_{#1,#2}}
\def\FMfrrn{ \FMfrvv{r}{n}}
\def\FMfron{ \FMfrvv{1}{n}}

\def\tFMfr{ \tFM^{\xfree} }
\def\tFMfrvv#1#2{ \tFMfr_{#1,#2}}
\def\tFMfrrn{ \tFMfrvv{r}{n}}
\def\tFRfron{ \tFMfrvv{1}{n}}

\def\cX{ \mathcal{X}}
\def\bcX{ \bar{\cX}}
\def\bcXfr{ \bcX^{\xfrm}}
\def\bcXfrv#1{ \bcXfr_{#1}}
\def\bcXfrr{ \bcXfrv{r}}
\def\bcXfrrn{ \bcXfrv{r,n}}

\def\bcXfrt{ \bcXfrv{2}}
\def\bcXfrtGv#1{ \bcXfrt(G_{#1}) }
\def\bcXfrtPv#1{ \bcXfrt(P_{#1}) }
\def\bcXfrtPk{ \bcXfrt(P_{k}) }

\def\tpX{ \calXr_3}
\def\tpXr{\calXr_{3,r}}
\def\tpXrfr{\calXr_{3,r}^{\xfr}}

\def\bcXGn{ \bcX_2(G_n) }

\def\bcXtfrv#1{ \bcX_{2,#1}^{\mathrm{fr}}}
\def\bcXtfrr{ \bcXtfrv{r}}

\def\SBim{\mathbf{SBim}}
\def\SBimv#1{ \SBim_{#1}}
\def\SBimn{ \SBimv{n}}

\def\hT{ \hat{T}}
\def\Db{ D^{\mathrm{b}}}
\def\DhT{ D^{\hT}}
\def\DbhT{ \Db_{\hT}}

\def\DTFMrn{\DhT(\FMrn)}
\def\DTsFMrn{\DhT(\FMrn)}
\def\DTFMfrn{\DhT(\FMfrrn)}
\def\DTsFMfrn{\DhT(\FMfrrn)}

\def\mZ{ Z }
\def\xik{ k }

\def\uD{ \underline{D}}
\def\uDhT{ \uD_{\hT}}
\def\xDTFM{ \uDhT(\FMrn)}
\def\xDTFM{ \DhT(\FMrn)}

\def\xDTFMfr{ \DhT(\FMfrrn)}

\def\mnsp{ \hspace{-0.2em}}
\def\dqt{ /\mnsp/ }
\def\dqtch{ \dqt_{\chi}}

\def\xeven{ \mathrm{even}}
\def\xodd{ \mathrm{odd}}

\def\xHm{ \mathcal{H}}
\def\xHmev{ \xHm^{\xeven}}
\def\xHmod{ \xHm^{\xodd}}
\def\xHmevod{ \xHm^{\xeven/\xodd}}

\def\xdev{ d_{\xeven}}
\def\xdodd{ d_{\xodd}}

\def\xMFBlv#1{\MF_B\bigl( #1 \bigr)}
\def\xMFBtlv#1{ \MF_{B^2}\bigl( #1 \bigr)}
\def\xMFBlzv#1{\MF_B\bigl( #1 ,0\bigr)}
\def\xMFBtlzv#1{ \MF_{B^2}\bigl( #1,0 \bigr)}

\def\DuhT{ D^{\hT}}
\def\yMFBlzv#1{\DuhT_B\bigl( #1 \bigr)}
\def\yMFBtlzv#1{ \DuhT_{B^2}\bigl( #1 \bigr)}

\def\xNS{ \mathrm{NS}}
\def\xNSv#1{ \xNS_{#1}}
\def\xNSn{ \xNSv{n}}

\def\yT{ \tau }

\def\Gn{ G_n }

\def\pibul{ \pi^{\bullet}}

\def\syt{ \mathrm{SYT}}
\def\sytv#1{ \syt_{#1}}
\def\sytn{ \sytv{n}}
\def\sytno{ \sytv{n-1}}

\def\lm{\lambda}
\def\lmp{ \lm' }

\def\yS{\sigma}

\def\xout{ \mathrm{out}}
\def\Cout{ C_{\xout}}

\def\KCHn{ K_{\CC^*}\bigl(D^{\Csq}(\FHilbn)\bigr) }

\def\dgD{ D_{\mathrm{dg}}}

\def\Dtot{ D_{\mathrm{tot}}}
\def\dce{ d_{\mathrm{CE}}}

\subsection{Hecke algebras}
Fix a positive integer $n$. The affine braid group on $n$ strands $\mfBraffn$ is generated by elements $\sg_1,\ldots,\sg_{n-1}$ and $\Dlto$ with a pictorial presentation
\begin{equation*}\label{AffineBraidGenerators}
\sigma_i =
\beginpicture
\setcoordinatesystem units <.5cm,.5cm>         
\setplotarea x from -5 to 3.5, y from -2 to 2    
\put{${}^{i+1}$} at 0 1.2      %
\put{${}^{i}$} at 1 1.2      %
\put{${}^{n}$} at -3 1.2
\put{${}^{1}$} at 3 1.2
\put{$\bullet$} at -3 .75      %
\put{$\bullet$} at -2 .75      %
\put{$\bullet$} at -1 .75      %
\put{$\bullet$} at  0 .75      
\put{$\bullet$} at  1 .75      %
\put{$\bullet$} at  2 .75      %
\put{$\bullet$} at  3 .75      %
\put{$\bullet$} at -3 -.75          %
\put{$\bullet$} at -2 -.75          %
\put{$\bullet$} at -1 -.75          %
\put{$\bullet$} at  0 -.75          
\put{$\bullet$} at  1 -.75          %
\put{$\bullet$} at  2 -.75          %
\put{$\bullet$} at  3 -.75          %
\plot -4.5 1.25 -4.5 -1.25 /
\plot -4.25 1.25 -4.25 -1.25 /
\ellipticalarc axes ratio 1:1 360 degrees from -4.5 1.25 center
at -4.375 1.25
\put{$*$} at -4.375 1.25
\ellipticalarc axes ratio 1:1 180 degrees from -4.5 -1.25 center
at -4.375 -1.25
\plot -3 .75  -3 -.75 /
\plot -2 .75  -2 -.75 /
\plot -1 .75  -1 -.75 /
\plot  2 .75   2 -.75 /
\plot  3 .75   3 -.75 /
\setquadratic
\plot  0 -.75  .05 -.45  .4 -0.1 /
\plot  .6 0.1  .95 0.45  1 .75 /
\plot 0 .75  .05 .45  .5 0  .95 -0.45  1 -.75 /
\endpicture
,
\qquad
\Delta_n =
~~\beginpicture
\setcoordinatesystem units <.5cm,.5cm>         
\setplotarea x from -5 to 3.5, y from -2 to 2    
\put{${}^{n}$} at -3 1.2
\put{${}^{1}$} at 3 1.2
\put{$\bullet$} at -3 0.75      %
\put{$\bullet$} at -2 0.75      %
\put{$\bullet$} at -1 0.75      %
\put{$\bullet$} at  0 0.75      
\put{$\bullet$} at  1 0.75      %
\put{$\bullet$} at  2 0.75      %
\put{$\bullet$} at  3 0.75      %
\put{$\bullet$} at -3 -0.75          %
\put{$\bullet$} at -2 -0.75          %
\put{$\bullet$} at -1 -0.75          %
\put{$\bullet$} at  0 -0.75          
\put{$\bullet$} at  1 -0.75          %
\put{$\bullet$} at  2 -0.75          %
\put{$\bullet$} at  3 -0.75          %
\plot -4.5 1.25 -4.5 -0.13 /
\plot -4.5 -0.37   -4.5 -1.25 /
\plot -4.25 1.25 -4.25  -0.13 /
\plot -4.25 -0.37 -4.25 -1.25 /
\ellipticalarc axes ratio 1:1 360 degrees from -4.5 1.25 center
at -4.375 1.25
\put{$*$} at -4.375 1.25
\ellipticalarc axes ratio 1:1 180 degrees from -4.5 -1.25 center
at -4.375 -1.25
\plot -2 0.75  -2 -0.75 /
\plot -1 0.75  -1 -0.75 /
\plot  0 0.75   0 -0.75 /
\plot  1 0.75   1 -0.75 /
\plot  2 0.75   2 -0.75 /
\plot  3 0.75   3 -0.75 /
\setlinear
\plot -3.3 0.25  -4.1 0.25 /
\ellipticalarc axes ratio 2:1 180 degrees from -4.65 0.25  center
at -4.65 0
\plot -4.65 -0.25  -3.3 -0.25 /
\setquadratic
\plot  -3.3 0.25  -3.05 .45  -3 0.75 /
\plot  -3.3 -0.25  -3.05 -0.45  -3 -0.75 /
\endpicture.
\end{equation*}
(strands are indexed from the right to the left)
modulo the braid relations\begin{gather}
\label{eq:brrel1}
   \sigma_i\; \sigma_{i+1}\; \sigma_i=\sigma_{i+1}\; \sigma_i\; \sigma_{i+1},\quad i=1,\; n-2,\\
\label{eq:brrel2}
   \sigma_i\; \sigma_j=\sigma_j\; \sigma_i,\quad |i-j|>1.
\end{gather}
and the affine relations
\begin{gather}
\label{eq:afrel1}
   \sigma_{n-1}\; \Delta_n\; \sigma_{n-1}\; \Delta_n=  \Delta_n\; \sigma_{n-1}\; \Delta_n\;\sigma_{n-1},\\
\label{eq:afrel2}
   \sigma_i\; \Delta_n=\Delta_n\; \sigma_i,\quad i<n-1,
\end{gather}
The ordinary braid group $\mfBrn$ is the quotient of $\mfBraffn$ modulo the relation $\Dlto = 1$. The homomorphism is induced by forgetting the flag pole in the above
picture and we use notation \(\fsur\) for the quotient map.

The affine Hecke algebra $\mfHaffn$ is defined as  a group ring $\CC[\qpmo](\mfBraffn)$ modulo the Hecke relations
\begin{equation}
\label{eq:herel}
(\sigma_i - q)(\sigma_i + q^{-1}) = 0,\qquad i = 1,\ldots,n-1.
\end{equation}

Fix a positive integer $\xr$ and introduce $\xr$ variables $\bxa = (\xa_1,\ldots,\xa_{\xr})$. The cyclotomic Hecke algebra $\mfHnr$ is defined as the quotient of $\CC[\qpmo,\bxa](\mfBraffn)$ modulo the affine Hecke relations and the cyclotomic relation
\begin{equation}
\label{eq:cyrel}
(\Dlto-\xa_1)\cdots (\Dlto-\xa_{\xr} ) = 0.
\end{equation}


To summarize, we have the following diagram
\begin{equation}
\label{eq:cmtsq}
\begin{tikzcd}[column sep=2cm, row sep=2cm]
\mfBraffn
\ar[r,"\fsur"]
\ar[d,"\xqtm"]
\ar[dr,"\fct"]
&
\mfBrn
\ar[d,dashed,"\xqtm"]
\\
\mfHaffn
\ar[r,"\fchifr"]
&
\mfHnr
\end{tikzcd}
\end{equation}
where $\fct := \fchifr\circ\xqtm$ and the dashed arrow appears only in the special case of $\xr=1$ when $\mfHno = \mfHn$. Moreover, for $\xrp< \xr$ consider $\xbap\subset \xba$, $\xbap = (a_1,\ldots a_{\xrp})$. Then there is an obvious homomorphism $\fchifrrrp\colon\mfHnr\rightarrow\mfHnrp$ with the properties $\fchifrvv{\xrp}{\xr''}\circ\fchifrrrp=\fchifrvv{\xr}{\xr''}$ and $\fchifrrrp\circ\fchifrr = \fchifrv{\xrp}$.


\subsection{A categorification of the affine Hecke algebra}

We use the following standard notations: $\mG = \GLnC$ is a group of $n\times n$ non-degenerate matrices; $\mfg = \gl(n,\IC)$ is its Lie algebra presented by $n\times n$ matrices; $\mB\subset\mG$ is the Borel subgroup presented by upper-triangular non-degenerate matrices; $\mfb\subset \mfg$ is the Borel subalgebra presented by upper-triangular matrices; $\mfn\subset\mfb$ is the nilpotent radical presented by strictly upper-triangular matrices.

We present a cotangent bundle over a flag variety $\mfB = \mG/\mB$ as another $\mB$-quotient:
$\TsB = (\mG\times\mfn)/B$, where $B$ acts as follows: $b\xact(g,Y) = (gb^{-1},\Adv{b}Y)$. The moment map of the $G$ action $\mu\colon\TsB\rightarrow\mfg$ takes the form $\mu(g,Y) = \Adv{g} Y$. We define an associated function $\yW\colon\mfg\times\TsB\rightarrow\IC$, $\yW(X,g,Y) = \Tr( X\Adv{g}Y)$.

All varieties appearing in our construction have an action of the `$(q,t)$-torus' $\Tqt = \Csq\times\Cst$. $\Tqt$ acts on $\mfg$ (and $\mfb$) and on the fibers of cotangent bundles $\rmTs\mfB$: $(a,b)\xact X = a^2 X$, $(a,b)\xact Y = a^{-2} b^2 Y$.


Our main category $\xCatn$ has several equivalent presentations as a category of equivariant matrix factorizations. In \cite{OblomkovRozansky16} we strive to define
the category of matrix factorizations:
\begin{equation}
\label{eq:cator}
\xCatn = \xMFTG\bigl(\mfg\times\TsB\times\TsB;\yW_2 - \yW_1\bigr),
\end{equation}
where $\yW_1$ and $\yW_2$ are functions associated with the first and of the second factors of $\rmTs\mfB$.

The $\mG$-equivariance is reduced to $\mB$-equivariance by rotating the first flag to the standard position:
$
\xCatn \cong \xMFTB\bigl(\mfg\times\mfn\times\TsB;\yW_2 -\Tr (X Y_1)\bigr).
$
The matrix $Y_1$ is bilinearly coupled with the strictly lower-diagonal part of the matrix $X$  in the term $\Tr(XY_1)$, hence the \Knrrp\ allows us to remove the factor of $\mfn$ while reducing $\mfg$ to $\mfb$:
\[
  \xCatn \cong \xMFTB\bigl(\mfb\times\TsB;\xW).
\]
Finally, invoking the quotient presentation of $\TsB$ we come to a category of $\mB\times\mB$-equivariant matrix factorizations over an affine variety $\cXtt$: 
\begin{equation}
\label{eq:affcat}
\xCatn \cong \xMFTBB\bigl(\cXtt;\xW\bigr)\footnote{The only category that is rigorously defined in our earlier papers is the  category  \(\mathbf{MF}_{B\times B}^{T_{qt}}(\calXr_2,\Wr)\), discussion here is a geometric motivation for the definition
of the latter category in \cite{OblomkovRozansky16}.},\qquad \cXtt=\cXtd.
\end{equation}

The index 2 in $\cXtt$ reminds of the presence of two factors of $\TsB$ in the original definition~\eqref{eq:cator}. These factors allow us  to define the monoidal structure of the category $\xCatn$ through a `bimodule-like' convolution, and we define in~\cite{OblomkovRozansky16} a homomorphism from the affine braid group on $n$ strands:
\[
\xPaff\colon\mfBraffn\rightarrow\xCatn.
\]
Moreover, we expect that  $\xCatn$ categorifies the $n$-strand affine Hecke algebra $\mfHaffn$ and the $\Csq$-equivariant K-theory of $\xCatn$ is isomorphic to $\mfHaffn$.
It is shown in \cite[Theorem 14.2]{OblomkovRozansky16} that the homomorphism $\xPaff$  is a part of a  commutative diagram
\[
\begin{tikzcd}
\mfBraffn \ar[r,"\xPaff"]
\ar[d,"\xqtm"]
&\xCatn
\ar[d,"\xKth"]
\\
\mfHaffn \ar[r,"\xpaff"]
&
\xKthq(\xCatn)
\end{tikzcd}
\]
Here $\xqtm$ is the quotient homomorphism which imposes quadratic Hecke relations, while $\xKth$ is the $K$-theory functor. 
In this paper we show that the homomorphism \(\xpaff\) is injective.

\subsection{Framing}
A transition from the affine Hecke algebra $\mfHaffn$ to the ordinary Hecke algebra $\mfHn$ was achieved in~\cite{OblomkovRozansky16,OblomkovRozansky17} through the introduction of a one-dimensional framing. In order to reach the cyclotomic Hecke algebra $\mfHnr$, we consider an $\xr$-dimensional framing. Namely, we consider an $n\times\xr$-dimensional vector space of linear maps
\[\xcVr = \Hom(\IC^{\xr},\IC^n).\]
The group $\mG$ acts on $\xcVr$ through its natural action on $\IC^n$. In addition, we consider the `$\xba$-torus' $\Tba = \IC^*_{a_1}\times\cdots\times\IC^*_{a_\xr}$ acting on $\xcVr$ through its natural action on $\IC^\xr$. Denote the combined torus $\Tqta = \Tqt\times\Tba$.


A triple $(X,Y,\xv)\in \mfg\times\mfg\times\xcVr$ is called stable, if
\begin{equation}
\label{eq:stab}
\IC\langle X,Y\rangle \,\xv(\CC^r) =\CC^n.
\end{equation}
Define the stable subvariety
$(\mfg\times\TsB\times\TsB\times\xcVr)_{\xstab}$ by the condition that both triples $(X,\mu_1,\xv)$ and $(X,\mu_2,\xv)$ are stable, and consider a `framed' version of the category $\xCatn$:
\[
  \xMFTaG\bigl((\mfg\times\TsB\times\TsB\times\xcVr)_{\xstab};\xW_2 - \xW_1\bigr).
\]

The rigorously defined category realizing such category is  an equivariant category of matrix factorizations over the stable framed version of $\cXtt$:
\begin{gather}
\label{eq:xtrfr}
\cXttsf = \bigl\{(X,g,Y,\xv)\in\cXtt\times\xcVr\,|\,\text{$(X,\Adv{g}Y,\xv)$ is stable}\bigr\},
\\
\label{eq:frcat}
\xCatnfrr:= \xMFTaBB\bigl(\cXttsf;\xW\bigr).
\end{gather}

The framing forgetting map  \(\fgt\colon\cXttsf\rightarrow \cXtt\) induces a pull-back functor
$\ffr\colon\xCatn\rightarrow\xCatnfrr$.
\begin{theorem}
The framing forgetting functor $\ffr$ commutes with the convolution of matrix factorizations.
\end{theorem}
This theorem allows us to define a homomorphism $\xPct\colon\mfBraffn\rightarrow \xCatnfrr$, which is a composition of $\xPaff$ and $\ffr$:
\[
\begin{tikzcd}
\mfBraffn
\ar[r,"\xPaff"']
\ar[rr, bend left=20,"\xPct"]
&
\xCatn
\ar[r,"\ffr"']
&
\xCatnfrr
\end{tikzcd}
\]
\begin{theorem}\label{thm:Ktheo}
There exists a homomorphism
\begin{equation}
\label{eq:kgrhom}
\xphict\colon \mfHnr\longrightarrow \xKthqa(\xCatnfrr)
\end{equation}
such that the following diagram is commutative:
\begin{equation}
\label{eq:simplesq}
\begin{tikzcd}[column sep=1.5cm,row sep=1.5cm]
\mfBraffn
\ar[r,"\xPct"]
\ar[d,"\fct"]
&
\xCatnfrr
\ar[d,"\xKth"]
\\
\mfHnr
\ar[r,"\xphict"]
&
\xKthqa(\xCatnfrr)
\end{tikzcd}
\end{equation}
If \(r=1\) then homomorphism \(\xphict\) is injective.
\end{theorem}

The injectivity argument  in the last theorem could be easily expanded to the case of general \(r\), the details will appear elsewhere. However we expect stronger
statement to be true:

\begin{conj}
The  homomorphism $\xphict$ is an isomorphism.
\end{conj}

The diagram~\eqref{eq:simplesq} is a  part of a larger commutative cube
\begin{equation}
\label{eq:cbcdgm}
\begin{tikzcd}[column sep=large,row sep=large]
&\mfBraffn \ar[rr,"\xPaff"]
\ar[dd,near end,"\xqtm"]
\ar[dl,"\fsur"']
\ar[dr,"\xPct"]
\ar[dddl,"\fct"']
&
&
\xCatn
\ar[dd,"\xKth"]
\ar[dl,"\ffr"]
\\
\mfBrn
\ar[dd,dashed,"\xqtm"']
\ar[rr,crossing over,dashed,near end,"\hPhi"]
&&
\xCatnfrr
\\
&
\mfHaffn \ar[rr,near end,"\xpaff"]
\ar[dl,near start,"\fchifr"']
&
&
\xKthq(\xCatn)
\ar[dl]
\\
\mfHnr
\ar[rr,"\xphict"]
&
&
\xKthqa(\xCatnfrr)
\ar[from=2-3,near end,crossing over,"\xKth"]
\end{tikzcd}
\end{equation}
in which the dashed arrows appear only when $\xr=1$.
This case is special,
because
various pieces of evidence compel us to promote the results of \cite{OblomkovRozansky16} to the
\begin{conj}
  There is an isomorphism of monoidal categories
  \[\mathcal{H}(S_n)\simeq \mathbf{MF}^{\xfr}_{n,1}.\]

\end{conj}

The LHS of the conjecture is the Hecke category that has many incarnations, the most algebraic of which is the category \(\SBimn\) of Soergel bimodules \cite{Soergel00}.

A categorification of the chain of homomorphisms $\fchifrrrp\colon\mfHr\rightarrow\mfHrp$ can be added to the diagram~\eqref{eq:cbcdgm}:
%
\[
\begin{tikzcd}[column sep=0.6cm,row sep=0.5cm]
&&\mfBraffn \ar[rrr,"\xPaff"]
\ar[ddd,"\xqtm"]
\ar[drr,"\xPct"]
\ar[ddddl,bend right=25,"\fct"]
\ar[dddddll, bend right=30]
\ar[ddr]
&
&
&
\xCatn
\ar[ddd,"\xKth"]
\ar[dl,"\ffr"]
\\
&&
&&
\xCatnfrr
\ar[dl,"\ffrrrp"]
\\
&&&
\xCatnfrrp
\\
&&
\mfHaffn \ar[rrr,"\xpaff"]
\ar[dl,near start,"\fchifr"']
&
&
&
\xKthq(\xCatn)
\ar[dl]
\\
&\mfHnr
\ar[rrr,"\xphict"]
\ar[dl,very near start,"\fchifrrrp"']
&
&
&
\xKthqa(\xCatnfrr)
\ar[from=2-5,crossing over,"\xKth"]
\ar[dl]
\\
\mfHnrp
\ar[rrr]
&&&
\xKthqa(\xCatnfrrp)
\ar[from=3-4,crossing over,"\xKth"]
\end{tikzcd}
\]
\vspace{1in}

\subsection{Two functors}
\label{sec:two-functors}

The method of proof is of interest in its own right because it clarifies the connection between our work and the work of Gorsky,
Negut and Rasmussen \cite{GorskyNegutRasmussen16}.  Indeed, above
mentioned work revolves around the study of the K-theory of the nested
Hilbert scheme, \(\FHilbn(\CC)\) in notations of
\cite{GorskyNegutRasmussen16}.


The flag Hilbert scheme $\FHilbn$ is the $r=1$ case of the flag instanton moduli space $\FMrn$. Consider an affine space
\[
\tFMfrrn = \bigl\{ (X,Y,\xv)\in\mfb\times\mfb\times\xcVr\;|\; \text{$(X,Y,\xv)$ is stable} \bigr\}
\]
which is a subscheme of $\cXttsf $ defined by the equation $g=1$. A subscheme \(\tFMrn\subset\tFMfrrn\) is defined by the commutativity condition $[X,Y]=0$.
The flag instanton moduli space $\FMrn$ and the `free' flag instanton moduli space \(\FMfrrn\) are defined as quotients of `tilded' spaces by the action of the Borel subgroup $B_n$:
\[
\FMrn = \tFMrn/B_n,\qquad \FMfrrn = \tFMfrrn/B_n.
\]
%

The space 
\(\FMfrrn\)
is smooth, so the category \(\xDTFMfr\) of
two-periodic complexes of coherent sheaves\footnote{Although the complexes are two-periodic, they still have a homological $\ZZ$-grading (rather than only a $\ZZ_2$-grading) coming from the weights of the action of $\Cst$.} is well-behaved. In contrast, the space \(\FMrn\) is very singular,
and we define the category \(\DTFMrn\) as a subcategory of objects
of \(\DTFMfrn\) with the relevant support condition,
see \autoref{sec:quot-spac-categ}.

In section~\ref{sec:funct-tild-tild} we construct two functors:
\[\tilde{\imath}^*: \DhT(\FMfrrn)\rightarrow
  \MF_{B^2}(\bcXtfrr,\Wr),\quad \tilde{\imath}_*:
  \MF_{B^2}(\bcXtfrr,\Wr)\rightarrow
D^{\hat{T}}(\FMrn)
  .\]
The tensor product of sheaves provides the category
\(\DTFMfrn\) with monoidal structure.
\begin{theorem}\label{thm:two-functors}
  The functor \(\tilde{\imath}^*\) is monoidal and for any
  vector bundle \(E\) and complex of sheaves $C$ on \(\FMfrrn\)  we have
  \[\tilde{\imath}_*\circ\tilde{\imath}^*(E\otimes C)=E\otimes
    \left(\tilde{\imath}_*\circ\tilde{\imath}^*( C)\right).\]
\end{theorem}

\subsection{Outline of the proof}
\label{sec:outline-proof}

Let us use notation \(\Delta_i\), \(i=1,\dots,n\) for the
\xBL\
elements of the affine braid group and \(\delta_i\)
are their images in the finite braid groups and Hecke algebras. The latter are also known as \xJM\ elements.  As it
is shown in \cite{OblomkovRozansky17}, the geometric counter-part of
the Berstein-Lusztig elements are the line bundles \(\mathcal{L}_i\)
over \(\calXr_2\):
\[\xPaff(\Delta_i)=\mathcal{L}_i.\]
We use the same notation for the corresponding line bundles over the
spaces \(\bcXtfrr\) and \(\FMrn\). In particular, we have
\[\phi^r(\delta_i)=\tilde{i}_*(\mathcal{L}_i).\]

A commutative subalgebra $\mfJMnr\subset\mfHnr$ generated by the \xJM\ elements \(\delta_i\) has a special basis.
The algebra $\mfHnr$  is semisimple:
\[\mfHnr=\oplus_{i\in I}\GL(V_i),\]
where \(I\) is the set of irreducible representations of the group \(S_n\ltimes \ZZ^n_r\) and
\(V_i\) are the corresponding vector spaces. The vector space \(V_i\) has a natural basis \(v_i^j\), \(j=1,\dots,d_i\), which we
describe later, and \(\delta_{i;j}\in \GL(V_i)\) is the diagonal matrix unit corresponding to the vector
\(v_i^j\). It is well-known that the algebra \(\mfJMnr\) is isomorphic to the subalgebra generated by the elements \(\delta_{i;j}\).
Let us denote the set of pairs \((i;j)\), \(i\in I\), \(j=1,\dots,d_i\) by \(\mathbf{I}\).

The locus $\FMrnT\subset \FMrn$ invariant with respect to the $\Tqta$-action is zero-dimensional and its points are naturally indentified with the elements of the set  \(\mathbf{I}\).
The scheme \(\FMrn\) has a \(\Tqta\)-equivariant affine cover:
\(\FMrn=\bigcup_{(i;j)\in \mathbf{I}}U_{(i;j)}\). Its charts
\(U_{(i;j)}\) are intersections of the affine charts \(\mathbb{A}_{(i;j)}\) of
\(\FMfrrn\).
Let us denote by
\([\mathcal{O}_{\mathbb{A}_{(i;j)}}]\) the corresponding class in K-theory
\(K_{\CC^*}(\FMfrrn)\). Our main technical result is
\begin{theorem}\label{thm:JMgeo}
  For every \((i;j)\) we have a non-zero \(C_{(i;j)}\in\Qqoa \) such that
  \[\tilde{\imath}_*\circ\tilde{\imath}^*([\mathcal{O}_{\mathbb{A}_{(i;j)}}])
    =C_{(i;j)}\cdot\tilde{\imath}_*(\phi^r(\delta_{(i;j)})).\]
\end{theorem}

In this paper we treat the case \(r=1\) of the above theorem. In this case the combinatorics of the affine cover is studied in the previous papers
\cite{OblomkovRozansky16,OblomkovRozansky17} and we rely on the results from these papers. For the general \(r\) the construction of the cover is analogous but
we postpone the detailed description for future publication.



Thus the restriction \(\phi^r|_{\mfJMnr}\) is an injective homomorphism and the
standard argument implies that the homomorphism \(\phi^r\) is also
injective.
We expect that the surjectivity of the homomorphism follows from the
variation of the localization theory and we hope to return to this part of the
theory in our future publications.

The functors \(\tilde{\imath}_*\), \(\tilde{\imath}^*\) should be compared to the same-named functors of~\cite{GorskyNegutRasmussen16}.
The key properties of the functors from our paper mimic the conjectural properties of the functors from \cite{GorskyNegutRasmussen16} and
it is very tempting to find a direct link to the setting of our previous paper \cite{OblomkovRozansky16}. The key step
would
be a functor between our support-defined derived category \(D_{\Tqta}(\FMon)\)
and the coherent DG category of \(\FMon =\FHilb_n(\CC^2)\) from \cite{GorskyNegutRasmussen16}.
In section~\ref{sec:dg-schem} we discuss potential challenges of constructing
such functor and describe a second order approximation to this functor.

Let us briefly discuss the structure of the text. In the section~\ref{sec:convolution-algebras} we recall the main constructions of
\cite{OblomkovRozansky16} concerning the convolution algebra structure on the
category of matrix factorizations. We explain why our earlier constructions
from \cite{OblomkovRozansky16} work in the case of higher dimesional framing and
show that after passing to \(K\)-theory we obtain a homomorphism from the
cyclotomic Hecke algebra to the convolution algebra of matrix factorizations.
In the section~\ref{sec:funct-tild-tild} we construct the functors \(\tilde{\imath}_*\), \(\tilde{\imath}^*\) and prove theorem~\ref{thm:two-functors}.
In the section~\ref{sec:jw-projectors} we explain the geometry behind the JW projectors and prove injectivity part of theorem~\ref{thm:Ktheo}.
Finally, in the section~\ref{sec:dg-schem} we discuss a relation between the
dg scheme structure of the flag Hilbert scheme and  the matrix factorizations that we study.

{\bf Acknowledgements:}
L.R. is very thankful to D.~Arinkin and I.~Losev for many illuminating discussions.
A.O.  is very thankful to R.~Bezrukavnikov and A.~Negu\c{t} for many illuminating
discussions. A.O. was partially supported
by NSF CAREER grant DMS-1352398.

\section{Convolution Algebras}
\label{sec:convolution-algebras}
In this section we recall the construction for the convolution algebra from \cite{OblomkovRozansky16} and explain why the pull-back map \(\fgt^*\) is the algebra homeomorphism.
We recall all necessary results from \cite{OblomkovRozansky16,OblomkovRozansky17}. Finally, we provide a proof for the Hecke relation inside our K-theoretic model.

\subsection{Affine Braid group realization}
\label{sec:affine-braid-group}


As explained in \cite{OblomkovRozansky16}, the category \(\xCatnx\) of~\eqref{eq:affcat} has a natural convolution product. An object \(\calF\in \xCatnx\)  is a quadruple
\[(M,D,\partial_l,\partial_r)\]
where \(M\) is \(\ZZ_2\)-graded \(\CC[\cXtt]\)-module and \(D\) is an endomorphism of \(M\) such that \(D^2=\Wr\); \(\partial_r,\partial_l: \Hom_{\CC[\calXr_2]}(M\otimes\Lambda^\bullet\frn,M\otimes\Lambda^{<\bullet}\frn)\) such that
\[\Dtot=D+\partial_l+\partial_r+\dce,\quad \Dtot^2=\Wr,\]
with \(\dce:
M\otimes U(\frn^2)\otimes \Lambda^\bullet(\frn^2)\rightarrow M\otimes U(\frn^2)\otimes \Lambda^{\bullet-1}(\frn^2)\) Chevalley-Eilenberg differential; we also impose
\(T^2\subset B^2\)-equivariance condition on all maps in the data of \(\calF\in \xCatnx\);
for details see \cite[section 3.4]{OblomkovRozansky16}.

The space \[\cXth:=\frb\times G^2\times \frn\] has a natural \(B^3\)-equivariant structure \cite[section 5.3]{OblomkovRozansky16} and there are maps
\(\bar{\pi}_{ij}:\cXth\rightarrow\cXtt\), see \cite[section 5.3]{OblomkovRozansky16}. The maps \(\bar{\pi}_{12},\bar{\pi}_{23}\) are not equivariant with respect to the middle \(B\) in the product
\(B^3\) but there is a natural way to endow the pull-back \(\bar{\pi}_{12}^*\otimes\bar{\pi}_{23}^*(\calF\boxtimes\calG)\), \(\calF,\calG\in \xCatnx\) with
\(B^3\)-equivariant structure \cite[section 5.4]{OblomkovRozansky16}, we denote this enhanced pull-back as \(\bar{\pi}_{12}^*(\calF)\otimes_B\bar{\pi}_{23}^*(\calG)\). It is shown in
\cite[corollary 5.7]{OblomkovRozansky16} that the binary operation:
\[\calF\bar{\star}\calG:=\bar{\pi}_{13*}\left(\CE_{n^{(2)}}(\bar{\pi}_{12}^*\otimes_B\bar{\pi}_{23}^*)(\calF\boxtimes\calG)^{T^{(2)}}\right)\]
defines an associative product on \(\xCatnx\).

In \cite[section 7.3]{OblomkovRozansky16} we define matrix factorizations \(\calCr_{\pm}^{(k)}\in \xCatnx \) \(k=1,\dots,n-1\), which represent elementary braids. It is shown in \cite[sections 9,10]{OblomkovRozansky16} that these matrix factorizations indeed satisfy the braid
relations:
\[\calCr_+^{(i)}\bar{\star}\calCr_+^{(i+1)}\bar{\star}\calCr_+^{(i)}\sim\calCr_+^{(i+1)}\bar{\star}\calCr_+^{(i)}\bar{\star}\calCr_+^{(i+1)},\]
\[\calCr_+^{(i)}\bar{\star}\calCr_-^{(i)}\sim \calCr_{\parallel}\]
where \(\sim\) is an explicit homotopy and \(\calCr_{\parallel}\) is the unit in the convolution algebra.


In \cite[section 3]{OblomkovRozansky17} we prove
\begin{theorem}
  There is a homomorphism:
$$\xPaff\colon \mfBraffn\rightarrow (
\xCatny
,\bar{\star}).$$
\end{theorem}

Given a matrix factorization $\mathcal{C}$ in $\xCatny$ and two characters $\xi,\tau: B\rightarrow \CC^*$ we define the twisted matrix factorization
$\mathcal{C}\langle\xi,\tau\rangle$ to be the matrix factorization $\mathcal{C}\otimes \CC_\xi\otimes\CC_\tau$. In these terms we have

\begin{theorem} \cite{OblomkovRozansky17} For any $i=1,\dots, n$ we have
$$\xPaff(\Delta_i)=\xPaff(1)\langle\chi_i,0\rangle.$$
\end{theorem}

The mutually commuting Bernstein-Lusztig (BL) elements $\Delta_i\in \mfBraffn$ are defined as follows:
\begin{equation}\label{BraidMurphy}
\Delta_i=\sigma_{i}\cdots \sigma_{n-2}\sigma_{n-1}
\Delta_n\sigma_{n-1}\sigma_{n-2}\cdots \sigma_{i}.
\end{equation}


\subsection{Hecke homomorphism}

In order to define the convolution in the category $\xCatnfrr$ of~\eqref{eq:frcat}, we have to introduce a stable framed version $\tpXrfr$ of the `triple' variety $\tpX$. The projections \(\bar{\pi}_{ij}\) extend naturally to
\begin{equation}
\label{eq:prjx}
\bar{\pi}_{ij}\colon \tpX\times\xcVr\rightarrow \cXtt\times\xcVr
\end{equation}
providing a definition of the convolution on the category \( \xMFBtlv{\cXtt\times\xcVr,\xW}\).

We define the stable framed triple variety $\tpXrfr$ as an open subset of $\tpX\times\xcVr$ which is an intersection of pre-images of the maps~\eqref{eq:prjx}:
\[\calXr_{3,r}^{\xfr}:=\bar{\pi}_{12}^{-1}(\calXr_{2,r}^{\xfr})\cap
  \bar{\pi}_{23}^{-1}(\calXr_{2,r}^{\xfr})\cap \bar{\pi}_{13}^{-1}(\calXr_{2,r}^{\xfr}).\]
%
%
%
%
Using the restrictions  of the pull-backs and push-forwards on the open sets we
define the convolution structure on the category \(\xMF^{\xfr}_{n,r}\).
  Moreover the pull-back along
the inclusion map \(\fgt\colon\cXttsf \rightarrow\cXtts\) respect the convolution:

\begin{lemma}\label{lem:homo}
  The pull-back map \(\fgt^*\) is a homomorphism of the convolution algebras.
\end{lemma}
\begin{proof}
  For \(r=1\) the statement is shown in \cite[lemma 12.2]{OblomkovRozansky16}, the case of general \(r\) is analogous.
  Indeed, according to the \cite[lemma 12.3]{OblomkovRozansky16} we need to analyze the behavior of the critical locus
  \(\cXttscr\),  which is defined by the ideal \(\xIcr\) of partial derivatives of
  \(\Wr\). We need to show that
  \[\bar{\pi}_{13}^{-1}(\cXttsf)\cap\bar{\pi}_{12}^{-1}(\cXttsfcr)
    \cap\bar{\pi}_{32}^{-1}(\cXttsfcr)=
  \bar{\pi}_{13}^{-1}(\cXttsf)\cap\bar{\pi}_{12}^{-1}(\cXttscr)
  \cap\bar{\pi}_{32}^{-1}(\cXttscr).\]
The argument identical to the one from \cite[lemma 12.2]{OblomkovRozansky16} implies the equality.
\end{proof}


The trivial line bundle over \(\tFMrn\) twisted by the \(B_n\)-character \(\chi_k\) descends to the
line bundle over \(\FMrn\) which we denote \(\mathcal{L}_k\). We also use the same notation \(\mathcal{L}_k\)
for the element  \(\calC_{\parallel}\langle\chi_i,0\rangle\) and we also use notation
\(\calO\) for \(\calC_{\parallel}\).
The torus \(\CC^*_{\mathsf{a}}\) acts on \(\calXr_{2,r}^{\xfr}\) and we use notation
\(\mathsf{a}_i\cdot\) for the operation on \(\xCatnfrr\) that twist matrix factorizations by the \(i\)-th character of \(\CC_{\mathsf{a}}\).

 The homomorphism $\fsur\colon\mfBraffn\rightarrow \mfBrn$ from the top line of the diagram~\eqref{eq:cmtsq} maps the \xBL\ elements of $\mfBraffn$ to the \xJM\ elements of $\mfBrn$:
%
$$ \mathfrak{fgt}(\Delta_n)=1,\quad \mathfrak{fgt}(\Delta_i)=\delta_i, \quad i=1,\dots,n-1.$$

In the case \(r=1\) we show in \cite[section 3]{OblomkovRozansky17} that the top square of the
cube~\eqref{eq:cbcdgm} commutes. For general \(r\) we can not draw the analogous cube because there is
no obvious analog of \(\Br_n\) for higher \(r\). The best approximation to the commuting square
from above is

\begin{prop} For any \(r\) there is a homotopy equivalence in category \(\xCatnfrr\):
  \[\mathrm{Kos}[\mathcal{L}_n\xrightarrow{\xv_n}\oplus_{i=1}^r\mathbf{a}_i\cdot\calO]\sim 0,\]
  where \(\xv_n\) is the vector of the last coordinates of \(\xv\) at the point
  \((X,g,Y,\xv)\in \cXttsf\) and \(\mathrm{Kos}\) is the Koszul complex.
\end{prop}
\begin{proof}
  The statement follows from the observation that the LHS of the homotopy equivalence is the
  Koszul complex for the skyscraper sheaf   at \(\phi_n=0\): by our definition
  \(\calXr_{2,fr,r}\) we have \(\phi_n\ne 0\).
\end{proof}

The K-theoretic version of the statement is the polynomial equation
\begin{equation}\label{eq:cyclo_rel}
  \prod_{i=1}^r(\mathcal{L}_n-a_i)=0,
\end{equation}
hence we prove

\begin{theorem}\label{thm:Kth}
There is a natural homomorphism of algebras
\(\xphict\colon\mfHnr
\rightarrow
\xKthqa(\xCatnfrr)
\)
which makes the square~\eqref{eq:simplesq} commutative.
%
\end{theorem}

\begin{proof}[Proof of theorem~\ref{thm:Ktheo}]
  The relations~\eqref{eq:brrel1}-\eqref{eq:afrel2} define the affine braid group, and
  it is shown in \cite{OblomkovRozansky16} and \cite{OblomkovRozansky17} that they
  are satisfied by the corresponding matrix factorizations of \(\xCatn\). Hence by lemma~\ref{lem:homo} these relations also hold for the corresponding matrix factorizations of
  $\xCatnfrr$.
  It is shown in \cite[section 14.2]{OblomkovRozansky16} that the Hecke relation~\eqref{eq:herel} is satisfied in K-theory. Finally, the cyclotomic relation~\eqref{eq:cyrel}
  is exactly the relation (\ref{eq:cyclo_rel}).

\end{proof}

\section{Functors \(\tilde{\imath}_*\) and \(\tilde{\imath}^*\)}
\label{sec:funct-tild-tild}
In this section we define the categories \(\DTFMrn \), \(D^{\hat{T}}(\mathrm{FM}_{r,n}^{\mathrm{free}})\) and the functors \(\tilde{\imath}_*\) and \(\tilde{\imath}^*\) relating them to $\xCatnfrr$. We show that these functors satisfy the projection formula and
\(\tilde{\imath}_*\) is monoidal.

\subsection{Quotient spaces and categories}
\label{sec:quot-spac-categ}

The space  \(\tFMfrrn\) is an open subset of the smooth space.
Moreover, the unipotent group \(U\subset B\) acts freely on this subset:
\begin{prop}
  The group \(U\) acts freely on \(\tFMfrrn\).
\end{prop}
\begin{proof}
If a point \((X,Y,\xv)\) has a non-trivial stabilizer then there is a non-zero strictly upper-triangular matrix  \(\mZ\in \frn\) such that
\[[X,\mZ]=[Y,\mZ]=0,\quad \mZ\xv=0.\]
Its kernel \(\ker \mZ\subset \IC^n\) is a proper subspace which contains the image of \(\xv\) and is preserved by \(X,Y\). Hence
\(\CC\langle X,Y\rangle \xv(V)\subset \ker \mZ\), which contradicts the stability condition.
\end{proof}

Thus  the quotient \(\tFMfrrn/U\) is free. In the case \(r=1\)  the  smoothness was shown by an explicit construction
of the smooth charts in \cite{OblomkovRozansky17a}. These charts have analogs for general \(r\) and we discuss them in a forthcoming paper.

The maximal torus \(T\subset B\) acts on the quotient \(\tFMfrrn/U\).
We define \(\chi\colon T\rightarrow \CC^*\) to be the determinant of a (diagonal) matrix.
The trivial bundle on the quotient \(\tFMfrrn/U\) with the equivariant
structure defined by \(\chi\) provides a polarization for the quotient and we consider
the GIT quotient  \(\bigr(\tFMfrrn/U\bigl)\dqtch T\). This quotient is smooth in view of the following proposition:
\begin{prop}
  Every point of \(\tFMfrrn/U\) is GIT stable.
\end{prop}
\begin{proof}
If a point \((X,Y,\xv)\) is not stable then there is
  a one parameter subgroup $\eta\colon\IC^*\rightarrow T$, $t\mapsto\eta(t) = (t^{\xik_1},\dots,t^{\xik_n})$,
   such that \(\sum_i \xik_i<0\) and the limit
  \(\lim_{t\to 0}\eta(t)(X,Y,\xv)\) exists. Let us denote by \(V_-\) the span of the vectors \(e_j\in\IC^n\) with \(\xik_j<0\) and by
  \(V_+\) is the span of vectors \(e_j\) with \(\xik_j\ge 0\).
  Then the image of \(\xv\) is a subspace of \(V_+\) and
  \(X,Y\in \End(V_+)\oplus\End(V_-)\) hence \(\CC\langle X,Y\rangle \xv(V)\subset V_+\).  Because of the assumption on $\eta$, the subspace
  \(V_+\) is proper contradicting the stability condition.
\end{proof}

Thus the space $
\FMfrrn:=\bigl(\tFMfrrn/U\bigr)
/\mnsp/_\chi T$ is smooth and we denote by \(\xDTFMfr\)
the derived category of the two-periodic  complexes of \(\hat{T}\)-equivariant coherent sheaves on
\(\FMfrrn\). There are two homology functors \(\xHmevod\) which assign to a complex
\((C_\bullet,d_\bullet)\in \xDTFMfr\)  the sheaves on \(\FMfrrn\):
\(\xHmev(C_\bullet,d_\bullet)=\ker \xdev/\xIm \xdodd\) and
\(\xHmod(C_\bullet,d_\bullet)=\ker \xdodd/\xIm \xdev\).

The defining ideal \(I_{r,n}\subset \CC\bigl[\tFMfrrn\bigr]\) for the subvariety
\(\tFMrn \)
is \(B\)-equivariant, hence it acts on
objects of
\(\xDTFMfr\).
Define the
category \(\xDTFM\)
as a subcategory of \(\xDTFMfr\) consisting of complexes \((C_\bullet,d_\bullet)\)
such that the action of the elements of \(I_{r,n}\) is homotopic to zero. In particular, for the objects of
\(\xDTFM\) the homology
\(\xHmevod(C_\bullet,d_\bullet )\) is scheme-theoretically supported on
\(\FMrn\subset\FMfrrn\).

Analogous
construction in K-theoretic setting was used in the book
\cite{ChrisGinzburg} and in paper \cite{BezrukavnikovRiche2012} to study
coherent sheaves on the Steinberg variety. The category \(\xDTFM\) has a natural monoidal structure
induced by the tensor product in the category \(\xDTFMfr\).

\subsection{Trace  functor}
\label{sec:trace-co-trace}
Denote
\[\cXttsfB = \cXttsf \cap \frb\times B\times\frn\times \xcVr.\]
There is a natural map
\(p\colon \cXttsfB\rightarrow \tFMfrrn\)
 which is defined by
\[p(X,g,Y,\phi)=(\Ad_{g}(X),Y,\phi).\]
The projection \(p\) induces the functor \(p^*\colon \yMFBlzv{\tFMfrrn}\rightarrow\yMFBtlzv{\cXttsfB}. \)

On the other hand the embedding \(i_B\colon\cXttsfB\rightarrow\cXttsf\) satisfies conditions of the construction of the
equivariant push forward from \cite[section 3.7]{OblomkovRozansky16}. Hence there is a well-defined functor:
\[i_{B*}\colon \yMFBtlzv{\cXttsfB}\rightarrow\xMFBtlv{\cXttsf,\xW}.\]

Since the space $\FMfrrn$ is a quotient of $\tFMfrrn$ by the action of $B$, there is a corresponding functor
\[\und{\mathrm{qt}}\colon
\yMFBlzv{\tFMfrrn}\rightarrow \xDTFMfr
\]
Finally, since \(\und{\mathrm{qt}}\) is an isomorphism,  we can define a
functor \(\tilde{\imath}^*:=i_{B*}\circ p^*:\)
\[\tilde{\imath}^*\colon \DhT(\FMfrrn)\rightarrow\xMFBtlv{\cXttsf,\xW}.\]

Similar to the discussion in the previous section we define a subcategory
\(\DTsFMrn \) of the category \(\DTsFMfrn\)
consisting of complexes with a homotopy connecting action of \(I_{r,n}\) with \(0\).

Let us also recall the construction of the functor \(\tilde{\imath}_*\) from \cite{OblomkovRozansky16}. The short outline of the
construction for \(r=1\) is available in \cite[section 1.3]{OblomkovRozansky17}. In the special case of \(r=1\)
the functor produces an isotopy invariant of the closure of the braid \cite{OblomkovRozansky16}.
Let \(j_e\) be an embedding
\[j_e\colon \tFMfrrn\rightarrow \cXttsf,\quad j_e(X,Y,\phi)=(X,e,Y,\phi).\]
Then the functor \(\tilde{\imath}_*\) is defined to be a pull-back \(j_e^*\). The functor lands in the category \(\DTsFMrn\):
\[\tilde{\imath}_*\colon\xMFBtlv{\cXttsf,\xW}\rightarrow \DTsFMrn.\]

\subsection{Proof of theorem~\ref{thm:two-functors}}
\label{sec:proof-theor-refthm:t}

The variety \(\tFMfrrn\subset\cXttsf\) is the zero locus of the ideal \(I_\Delta\subset\CC[\frb\times G\times\frn]\) with the generators
\(g_{ij}\), \(i>j\). The generators form a regular sequence and the Koszul complex \(\mathrm{K}^{\Wr}(I_\Delta)\in \MF_{B^2}(\calXr_{n,r})\)
is the unit \(\mathds{1}_n\) in the convolution algebra \cite[section 7.1]{OblomkovRozansky16}. Thus for any pair of vector bundles \(E,F\) on $\FMfrrn$
we have
\[\tilde{\imath}^*(E)\bar{\star}\tilde{\imath}^*(F)=\tilde{\imath}^*(E\otimes F)\]
because \(\mathds{1}_n\bar{\star}\mathds{1}_n=\mathds{1}_n\). Since general sheaf on $\FMfrrn$ have a finite projective resolution, the
previous observation implies that the functor \(\tilde{\imath}^*\) is monoidal.

The last statement of the theorem follows immediately from the construction of the functors \(\tilde{\imath}_*,\tilde{\imath}^*\).

\section{Geometry of JW projectors}
\label{sec:jw-projectors}

In this section we show the injectivity of the homomorphism \(\phi\). As a by-product of our proof we provide a geometric construction for the JW
projectors in the  Hecke algebra. First we recall algebraic properties of the JW projectors.

\subsection{Inductive construction of the quasi-projectors}
\label{sec:induct-constr-quasi}
The \xJM\ subalgebra of the (ordinary) Hecke algebra $\yJMn\subset\mfHn$ has a basis of idempotents, whose elements are labelled by standard Young tableaux.
A tableaux \(\yT\) of size \(n\) is a Young diagram
with squares labeled by elements of the set \(\{1,\dots,n\}\),
the labels increasing along the rows and columns. In particular, for a standard Young tableau \(\yT\in\mathrm\sytn\)  the union of squares with
labels smaller or equal \(k\) is a Young diagram of size \(k\), and we denote this (sub-)diagram as \(D_k(\yT)\).

\def\square{ \Box }

The outer corners of a Young diagram \(\lambda\) is the set of squares \(C(\lambda)\) such that if
\(\square\in \Cout(\lambda)\) then the union \(\lambda\cup \square \) is a Young diagram. The content
of a square \(\square\) is the difference of the row number and the column number:
\[c(\square)=\mathrm{row}(\square)-\mathrm{col}(\square).\]

The algebra \(\JM_n\) is generated by the JM elements \(\delta_i\), \(i=1,\dots,n-1\).
To a tableau \(\yT\) we attach the following polynomial of JM elements:
\begin{equation}\label{eq:JMproj}
  Q_{\yT}=\prod_{k=1}^{n-1}\frac{\prod_{\square\in \Cout(D_k)}(\delta_{n-k}-q^{2c(\square)})}{(\delta_{n-k}-q^{2c(\square_k)})},
  \end{equation}
where \(\square_k\) is the square with the label \(k\) and \(D_k=D_k(\tau)\).

\begin{prop}\cite{AistonMorton98}
  The elements \(Q_{\yT}\), \(\yT\in \sytn\) form a \(\CC(q)\) basis of \(\yJMn\).
  The elements \(Q_T\) are quasi-idempotent and
  \[\delta_{n-k}\cdot Q_{\yT}=q^{2c(\square_k)}Q_{\yT},\]
  where \(\square_k\) is the box the label \(k\).
\end{prop}

The proposition also implies that the elements
\(Q_{\yT}\) are proportional to the diagonal matrix elements in the decomposition:
\begin{equation}\label{eq:dec}
  \mfHn=\bigoplus_{i\in I}\GL(V_i),
  \end{equation}
where the sum is over the set $I$ of irreducible representations of \(S_n\).

\subsection{Sheaves on the flag Hilbert scheme}
\label{sec:sheaves-flag-hilbert}
In this section we define and discuss the category \(\DTsFMrn \) in
the simplest case \(r=1\), the case of general \(r\) is similar and will
be discussed elsewhere.

We use notation \(\FHilbfrn \) for \(\FMfron \) to make our notation close to both
\cite{OblomkovRozansky16} and \cite{GorskyNegutRasmussen16}.
It is shown in \cite[section 13.1]{OblomkovRozansky16} that the space \(\FHilbfrn\) has
a \(B\)-equvariant affine cover. The details of the cover are discussed in  \cite{OblomkovRozansky17a} where
we introduce a natural labeling set \(\xNSn\) for the affine charts:
an affine space \(\mathbb{A}_S\)  corresponding to a label \(\mathbb{A}_{S}, S\in \xNSn\).

The standard Young tableaux form a subset \(\sytn\subset\xNSn\),
and the union
\[\mathbb{A}_{\sytn}:=\bigcup_{\yT\in \sytn\subset\xNSn}\mathbb{A}_{\yT},\]
is an neighborhood of \(\FHilbn\) inside \(\FHilbfrn\), see \cite[section 6]{OblomkovRozansky17a}.
 The key property of
\(\mathbb{A}_{\sytn}\) is that the $\Csq$-invariant locus \(\bigl(\mathbb{A}_{\sytn}\bigr)^{\Csq}\) is zero-dimesional.

The \(\Tqt\)-fixed locus inside \(\FHilbn\) is equal to the \(\Csq\)-fixed locus and it is zero-dimesional.
A point of of \(\FHilbnC\) is a chain of monomial ideals:
\[I_1\supset I_2\supset\dots\supset I_n,\]
with \(\dim(\CC[x,y]/I_n)=n\). The quotient \(I_{k}/I_{k+1}\) is spanned by a monomial \(x^{i_k}y^{j_k}\).
The points \((i_k,j_k)\), $1\leq k \leq n$ determine
 a standard Young
tableaux. Thus the fixed points \(\FHilbnC\) are naturally labeled by the standard Young tableaux and
we denote by \(I_{\bullet,\yT}\), \(\yT\in \sytn\) the corresponding chain of ideals.

Each affine space \(\mathbb{A}_{\yT}\), \(\yT\in SYT_n\) has the chain \(I_{\bullet,\yT}\) as its \(\Csq\)-fixed center.
The weight of the $\Csq$-action on the stalk of \(\mathcal{L}_i\) is \(2c(\square_k)\) where \(\square_k\in
\yT\)
is the square labelled by $k$ in the tableaux \(\yT\). The variety \(\mathbb{A}_{\yT}\) is an affine space,
hence there is a relaiton in \(K\bigl(\FHilbfrn\bigr) \):
\[[\mathcal{L}_i\otimes \mathcal{O}_{\mathbb{A}_T}]=q^{2c(\square_k)}\cdot [\mathbb{A}_T].\]

\subsection{Induction functors}
\label{sec:induction-functors}

Before we proceed to the proof of the theorem~\ref{thm:Kth} let us recall the induction functor
from \cite{OblomkovRozansky16}. Denote for brevity $\Gn = \GL(n)$.
The Lie algebra of the standard parabolic subgroup \(P_k\subset \Gn\) is generated by the Borel subalgebra \(\frb\) and by the under-diagonal elements \(E_{i+1,i}\), \(i\ne k\).
Define  \(\calXr_2(P_k):=\frb\times P_k \times \frn\) and denote
\(\bcXGn=\calXr_2\) in order to make the $n$-dependence explicit. There is a natural embedding \(\bar{i}_k:\calXr_2(P_k)\rightarrow\bcXGn\) and
a natural projection \(\bar{p}_k:\calXr_2(P_k)\rightarrow\calXr_2(G_k)\times\calXr_2(G_{n-k})\). The embedding \(\bar{i}_k\)
satisfies
the conditions for existence of the push-forward functor for matrix factorizations with the potential $\xW$
and we define the induction functor:
\[\overline{\ind}_k:=\bar{i}_{k*}\circ \bar{p}_k^*: \MF_{B_k^2}(\calXr_2(G_k),\Wr)\times\MF_{B_{n-k}^2}(\calXr_2(G_{n-k}),\Wr)\rightarrow\MF_{B_n^2}(\calXr_2(G_{n}),\Wr)\]

The stability condition~\eqref{eq:stab} defines an open subspace $\bcXfrtPk\subset \frb\times P_k \times \frn\times
\IC^n$.
A natural projection map \(\bar{p}_k\colon\bcXfrtPk\rightarrow\calXr_2(G_k)\times\bcXfrtGv{n-k}\) and
an embedding \(\bar{i}_k\colon \calXr_{2}^{\xfr}(P_k)\rightarrow\calXr_{2}^{\xfr}(G_{n})\)  define the induction functor:
\[\overline{\ind}_k:=\bar{i}_{k*}\circ \bar{p}_k^*\colon \MF_{B_k^2}\bigl(\calXr_2(G_k),\Wr\bigr)\times\MF_{B_{n-k}^2}\bigl(\bcXfrtGv{n-k},\Wr\bigr)\rightarrow
\MF_{B_n^2}\bigl(\bcXfrtGv{n},\Wr\bigr)\]

It is shown in  \cite[proposition 6.2]{OblomkovRozansky16} that the functor \(\overline{\ind}_k\) is the homomorphism of the convolution algebras:
\[\overline{\ind}_k(\mathcal{F}_1\boxtimes\mathcal{F}_2)\bar{\star} \overline{\ind}_k(\mathcal{G}_1\boxtimes\mathcal{G}_2)=
  \overline{\ind}_k(\mathcal{F}_1\bar{\star}\mathcal{G}_2\boxtimes\mathcal{F}_2\bar{\star}\mathcal{G}_2).\]

\subsection{Proof of theorem~\ref{thm:JMgeo}}
\label{sec:proof-theor-refthm:g}
We proceed by induction on \(n\).


Consider a projection  \[\pi\colon\FHilbfrn\rightarrow\FHilbfrno\] which acts by removing the top rows of the matrices
$X,Y\in\mfb$ and the top entry of the vector $\xv\in\IC^n$. It induces a pull-back functor $\pi^*\colon \DhT(\FHilbfrno)\rightarrow\DhT(\FHilbfrno)$. We define an associated functor \[\pibul\colon\DhT(\FHilbno)\rightarrow\DhT(\FHilbn).\]
A pull-back $\pi^*(C)$ of a complex $C\in\DhT(\FHilbno)$ is a complex of sheaves on $\FHilbfrn$ with the homological support at $\pi^{-1}(\FHilbno)$. The flag Hilbert scheme $\FHilbn$ is a subvariety of $\pi^{-1}(\FHilbno)$ defined by zeroes of the functions
%
\[f_i(X,Y):=[X,Y]_{1,i},\quad i=2,\dots,n,\]
where \(X\in\frb,Y\in \frn\) are the coordinates on \(\FHilbfrn\). Thus we define
\[\pi^\bullet(C):=\pi^*(C)\otimes \left(\bigotimes_{i=2}^n\mathrm{K}(f_i)\right),\]
where \(\mathrm{K}\) denotes the Koszul complex.

The functor \(\tilde{\imath}_*\circ\tilde{\imath}^*\) intertwines \(\pi^*\) and \(\pi^\bullet\):
\[\pi^\bullet\circ  \tilde{\imath}_*\circ\tilde{\imath}^*=\tilde{\imath}_*\circ\tilde{\imath}^*\circ\pi^*. \]
Indeed, it is enough to check it for the structure sheaf \(\mathcal{O}_{\FHilbfrn}\). Then the result follows from the observation that
\[\tilde{\imath}_*\circ\tilde{\imath}^*(\mathcal{O}_{\FHilbfrn})=\bigotimes_{2\leq i\leq j\leq n}\mathrm{K}(f_{ij}),\]
where \(\mathrm{K}\) is the Koszul complex
and
\(f_{ij}=[X,Y]_{ij}\).

Our strategy is to compare the pull-backs along the projection \(\pi\)
with the inclusion map \(\mathfrak{ind}\colon \Br_{n-1}\rightarrow \Br_n\) where
\(\Br_{n-1}\) is the subgroup of \(\Br_n\) generated by \(\sigma_i\), \(i=2,\dots,n-1\). The key property of the \(\overline{\ind}_k\) is that
\[\overline{\ind}_1(\mathds{1}_1\boxtimes \Phi_{n-1}(\beta))=\Phi_n(\mathfrak{ind}(\beta)).\]

Another formula required for the proof is similar to the one used in the proof of the Markov move \cite{OblomkovRozansky16}:
\begin{equation}\label{eq:pi_bul}
  \tilde{\imath}_*(\overline{\ind}_1(\mathds{1}_1\otimes C))=\pi^\bullet(\tilde{\imath}_*(C)).
\end{equation}
In order to prove it we recall the  push-forward from \cite[section 3.2, section 3.7]{OblomkovRozansky16}.
We are interested in the details of the push-forward \(\bar{i}_{1*}\) that appears in the definition of the functor \(\overline{\ind}_1\).

The subspace \(\calXr_2(P_1)\subset \calXr_2(G_n)\) is defined by the equation \(g_{i1}=0\), \(i=2,\dots,n\) where
\(X,g,Y\) are the coordinates on \(\calXr_2(G_n)=\frb_n\times G_n\times\frn_n\).
It is explained \cite[section 6.3]{OblomkovRozansky16} the push-forward \(i_{1*}(\calF)\)
of \(\calF\in \MF_{B^2}\bigl(\calXr_2(P_1)\bigr)\) is the matrix factorization:
\[\tilde{\calF}\otimes \mathrm{K}^{\delta_1\Wr}(g_{i1}),\]
where \(\tilde{\calF}\) is a canonical extension of \(\calF\) to \(\calXr_2(G_n)\):
\(\tilde{\calF}|_{\calXr_2(P_1)}=\calF\)
(see details in \cite[section 6.3]{OblomkovRozansky16}) and \(\delta_1\Wr=\Wr-\Wr|_{g_{12}=\dots=g_{1n}=0}\).

If we set \(\calF=p_1^*(\mathds{1}_1\boxtimes\calG)\), \(G\in \MF_{B^2}\bigl(\calXr_2(G_{n-1})\bigr)\) in last formula we get
\[\tilde{\imath}_*(\overline{\ind}(\calG))=j_e^*(i_{1*}(\calF))=j_e^*(\tilde{\calF})\otimes j_e^*(\mathrm{K}^{\delta_1\Wr}(g_{i1})).\]
Now let us observe that the first factor in the last formula is \(j_e^*(\mathds{1}_1\boxtimes\calG)=
\pi^*(j_e^*(\calG))\). On the other hand \(j_e^*(\delta_1\Wr)=0\) and \(j_e(\mathrm{K}^{\delta_1\Wr}(g_{i1}))=
\mathrm{K}\left(j_e^*(\frac{\partial\Wr}{\partial g_{i1}})\right)\). Since
\(j^*_e(\frac{\partial\Wr}{\partial g_{i1}})(X,Y)=[X,Y]_{1i}\) the formula (\ref{eq:pi_bul}) follows.

For a Young tableau \(\yT\in\sytn\) let \(\yT'\in \sytno\) be the tableau constructed by removing the box labelled \(n\) from \(\yT\).
Denote by $\lmp$ the Young diagram corresponding to the tableau $\yT'$.
The formula \ref{eq:JMproj} shows that
\[Q_{\yT}=(\mathfrak{ind}(Q_{\yT'}))\cdot \prod_{\square\in \Cout(\lmp)\setminus \yT}(\delta_1-q^{2c(\square )}). \]

Let us denote the factor in the last formula by \(P(\delta_1)\) then we have the following equalities in the
K-theories:
\begin{multline*}
  \tilde{\imath}_*(\phi_n(Q_{\yT}))=P(\mathcal{L}_1)\tilde{\imath}_*(\phi_n(\mathfrak{ind}(Q_{\yT'})))=
  P(\mathcal{L}_1)\tilde{\imath}_*(\overline{\ind}_1(\mathds{1}_1\otimes\phi_{n-1}(Q_{T'})))=\\
  P(\mathcal{L}_1)\pi^\bullet(\tilde{\imath}_*(\phi_{n-1}(Q_{\yT'})))=C_{\tau'}P(\mathcal{L}_1)
  \pi^\bullet(\tilde{\imath}_*\circ\tilde{\imath}^*([\mathcal{O}_{\mathbb{A}_{\yT'}}]))=
  C_{\yT'}\tilde{\imath}_*\circ\tilde{\imath}^*(P(\mathcal{L}_1)\pi^*([\mathcal{O}_{\mathbb{A}_{T'}}]))=\\
  C_{\yT}\tilde{\imath}_*\circ\tilde{\imath}^*([\mathcal{O}_{\mathbb{A}_{\yT}}]).
\end{multline*}

In the last equality is because \(\pi^{-1}(\mathbb{A}_{T'})=\bigcup_{\yS} \mathbb{A}_{\yS}\) where \(\yS\in \sytn\) such that \(\yS\setminus \yT'\in \Cout(\yT')\).
If \(\yS\setminus \yT'=\square\) then \((\mathcal{L}_1-q^{2c(\square)})\mathbb{A}_\yS=0\). Hence
\[P(\mathcal{L}_1)\pi^*(\mathbb{A}_{\yT'})=\prod_{\square\in \Cout(\yT')\setminus \yT}(q^{2c(\square')}-q^{2c(\square)})\cdot\mathbb{A}_{\yT},\]
where \(\square'=\yT\setminus \yT'\).

\subsection{Proof of Theorem~\ref{thm:Ktheo}}
\label{sec:inj}
The elements \(\tilde{\imath}_*\circ\tilde{\imath}^*(\mathbb{A}_{\yT})\),
\(\yT\in \sytn\) are linearly independent in the group \(\KCHn\). Indeed, the affine subvarieties
\(U_{\yT}=\FHilbn\cap \mathbb{A}_\yT\) form \(\sCsq\)-equivariant \v{C}ech cover of
\(\FHilbn\) and variety \(U_\yT\) has only one \(\sCsq\)-fixed point, hence the skyscraper at the torus-fixed point
\(\delta_\yT\in\KCHn \) are non-trivial. Hence a non-trivial relation
between the elements \(\tilde{\imath}_*\circ\tilde{\imath}^*(\mathbb{A}_\yT)\) contradicts
the non-triviality of \(\KCHn \) after taking product with elements \(\delta_\yT\).

Thus the homomorphism \(\tilde{\imath}_*\circ\phi: \yJMn\to\KCHn\) is injective, hence
the restriction of \(\phi\) on \(\yJMn\) is also injective. On the other hand, if the homomorphism
\(\phi\) has a kernel, the kernel is a non-trivial subsum of LHS of (\ref{eq:dec}). That would contradict
the injectivity of the homomorphism of \(\phi|_{\yJMn}\) since \(\yJMn\) is generated by the diagonal entries of
matrix algebras of LHS (\ref{eq:dec}).

\section{DG schemes vs Matrix Factorizations}
\label{sec:dg-schem}

\subsection{DG schemes and weak DG schemes}
\label{sec:instantons}

In this paper we take the perspective of \cite{GorskyNegutRasmussen16} and ignore all topological issues of the theory of dg schemes. The appendix of \cite{GorskyNegutRasmussen16} was particularly
useful for writing our paper. The dg structure on \(\FHilbn(\CC^2)\) differs slightly from the structure from \cite{GorskyNegutRasmussen16} because of the way
we deal with group-equivariant structure. We define the dg structure on \(\tFMrn\)
as the one induced by
\[\mathcal{O}_{\tFMrn}=S_{\calXr_{2,r}}[\mathfrak{n}^{\vee}\rightarrow \mathcal{O}]\]
where \(\mathfrak{n}^\vee\) is the trivial bundle over \(\calXr_{2,r}\) with fiber \(\mathfrak{n}^\vee\) and the differential is matrix the strictly upper-triangular part  of   \([X,Y]\). The term \(\mathfrak{n}^\vee\) is in the homological degree \(-1\) hence yields the
algebra \(\Lambda^*(\mathfrak{n}^\vee)\).


Similarly, an object  of \(\dgD(\tFMrn)\)  over the quasi-affine dg scheme \(\tFMrn\) is a
pair $\bigl( (\mathcal{P}^\bullet,d_{\mathcal{P}}), \lambda \bigr)$, where $\mathcal{P}^\bullet$ is a two-periodic complex of quasi-coherent sheaves, $d_{\mathcal{P}}$ being its differential, while $\lambda$ is a map
\[\lambda:\mathfrak{n}^\vee\otimes \mathcal{P}^\bullet\rightarrow\mathcal{P}^{\bullet-1},\]
such that \(\lambda^2=0\) in the sense of the commutativity of the following diagram:
\begin{equation}\label{eq:anticomm}\begin{tikzcd}
  &\mathfrak{n}^\vee\otimes\mathcal{P}^{\bullet-1}\arrow[rd,"\lambda"]&\\
  \mathfrak{n}^\vee\otimes\mathfrak{n}^\vee\otimes\mathcal{P}^{\bullet-2}\arrow[ru,"\lambda^{(1)}"]\arrow[rd,"\lambda^{(2)}"]& &\mathcal{P}^\bullet\\
  &\mathfrak{n}^\vee\otimes\mathcal{P}^{\bullet-1}\arrow[ru,"-\lambda"]&
  \end{tikzcd}
\end{equation}
We also require the homomorphisms \(\lambda_{ij}\) to be compatible with the differential \(d_\mathcal{P}\):
\[[d_\mathcal{P},\lambda_{ij}]=[X,Y]_{ij}.\]

Thus the homomorphisms \(\lambda_{ij}\) are homotopies that connect \([X,Y]_{ij}\) with zero and the condition (\ref{eq:anticomm}) tells us that these homotopies anti-commute.
Let us weaken the condition of commutativity of (\ref{eq:anticomm}) and obtain a bigger category \(D'(\tFMrn)\) where the objects are triples
\[((\mathcal{P},d_\mathcal{P}),\lambda,\chi),\quad \chi:\mathfrak{n}^\vee\otimes\mathfrak{n}^\vee\otimes \mathcal{P}^{\bullet-2}\rightarrow\mathcal{P}^{\bullet},\]
\(\chi\) being a homotopy between the composition of the top and bottom arrows in the diamond
(\ref{eq:anticomm}). It is reasonable to call this category
a derived category of coherent sheaves on the {\it weak dg scheme.}

The central category in the main body of this paper as well as in
\cite{OblomkovRozansky16},\cite{OblomkovRozansky17} is the subcategory  of
the category
\(\yMFBlzv{\tFMfrrn}\)
generated by the objects that are obtained
by the pull-back functor \(j_e^*\) from elements of \(\MF_{B^2}(\calXr_2,\Wr)\).
Since the elements of this subcategory are supported on \(\tFMrn\)
there is a natural functor from this subcategory to the \(B\)-equivariant version
\(D'_B(\FMrn)\) of \(D'(\tFMrn)\). On the other hand there no obvious functor to a
\(B\)-equivariant version of \(\dgD(\tFMrn)\)


It appears though there is a way to relate our categories of matrix factorizations to the category of sheaves on the dg-scheme \(\FHilbn\)
and we discuss our proposal in the next subsection. To be more precise we discuss
the relation with sheaves on the dg-scheme of commuting matrices, we choose
to ignore the stability condition to simplify our notations.

\subsection{Matrix factorizations}
\label{sec:matr-fact}

Let \(B_-\subset G\) be a subgroup of lower-triangular matrices, we denote by \(G^\circ\subset G\) the open subset \(B_-\cdot B\) and respectively \(\calXr^\circ_{r,n}=\calXr_{r,n}\cap (\frb\times G^\circ\times\frn\times\xcVr)\). The open embedding
\(u:\calXr^0_{r,n}\rightarrow\calXr_{r,n}\) induces:
\[u^*\colon\MF_{n,r}\rightarrow\MF_{n,r,\circ}:=\MF_{n}(\calXr^\circ_{r,n}).\]

We expect the category \(\MF_{n,r,\circ}\) to be closely related to the category \(D^b(FM_{n,r})\). Our optimism stems out of the Koszul duality in the following setting.
Let us introduce the linear version of our potential:
\[\xWl(X,Z,Y)=\Tr(X[Y,Z]).\]
The linear Koszul duality then implies an isomorphism of categories:
\[\MFs(\frb\times \frn_-\times\frn,\xWl)\simeq D^{b}(\mathcal{C}),
\]
where \(\mathcal{C}\) is a commuting variety. On the other hand we can reduce the equivariant structure  in our original matrix factorization:
\[\MF_{B^2}(\calXr^{\circ}_{r,n},\Wr)=\MF_{B}(\calXr^{\bullet}_{r,n},\Wr),\]
where \(\calXr^{\bullet}_{r,n}=\calXr_{r,n}\cap (\frb\times U_-\times\frn\times \xcVr)\).

Ignoring the equivariant structure  we could provisionally expect some close relation between the two categories:
\[\MF(\frb\times\frn_-\times\frn,\xWl)\stackrel{?}{=}\MF(\frb\times U_-\times\frn,\Wr).\]
At the moment we do not have a precise version of the equality above, we have a first order approximation to the equality that we discuss below.

First we identify \(U_-\) and \(\frn_-\) by the exponential map. Under this identification we see that
\[\Wr^{lin}(X,Z,Y)=\Wr(X,Z,Y)\mbox{ mod } \mathfrak{m}(Z)^2,  \]
where \(\mathfrak{m}(Z)\) stands for the ideal generated by \(Z_{ij}\). It also clear that these potentials differ modulo \(\mathfrak{m}(Z)^3\). Nevertheless the corresponding categories
are equivalent in the following sense.

Let us introduce the notation \(\calXr^{\sim k}_{n,r}\) for the \(k\)-th infinitesimal neighborhood of the \(Z=0\) locus inside \(\calXr^\bullet_{r,n}\). As we discussed before the restrictions
of \(\xWl\) and \(\Wr\) on \(\calXr_{r,n}^{\sim 1}\) coincide, a strengthening of this observation is the following:
\begin{lemma}
  There is a natural isomorphism:
  \[\MF(\calXr_{r,n}^{\sim 2},\xWl)=\MF(\calXr_{r,n}^{\sim 2},\Wr).\]
\end{lemma}
\begin{proof}
  First we do the following  change of variables:
  \[(X,Z,Y)\mapsto (X,Z,Y-\frac12[Y,Z]_{++}).\]

  Under this change of variables the potential \(\Wr\) turns into the potential:  \begin{multline*}\Wr'(X,Z,Y)=\Tr(X\ad_Z(Y))+\frac{1}{2}\Tr(X\ad_Z([Y,Z]_-))-\frac{1}{4}\Tr(X\ad_Z^2([Y,Z]_{++}))\\+\sum_{k\ge 3}\frac{1}{k!}\Tr(X\ad_Z^3(Y-\frac12[Y,Z]_{++})).\end{multline*}

  The last two terms are cubic \(Z\), on the other hand the second is quadratic and it could be rewritten in the form
  \(\Tr([X,Z]_+[Y,Z]_-)\). Since the second term inside \(\Tr\) in the last expression is proportional to \(\frac{\partial \Wr}{\partial X}\) hence if \((D',M)\in \MF(\bullet,\Wr')\) then
  the modified differential:
  \[D''=D'+\sum_{ij}[X,Z]_{ij}\frac{\partial D'}{\partial X_{ij}}+\sum_{ij,kl}\frac{\partial^2 D'}{\partial X_{ij}\partial X_{kl}}[X,Z]_{ij}[X,Z]_{kl},\]
  module \(\mathfrak{m}(Z)^3\) squares to
  \begin{multline*}
    (D')^2+\sum_{ij}[X,Z]_{ij}[D',\frac{\partial D'}{\partial X_{ij}}]+\sum_{ij,kl}\left(\frac{\partial D'}{\partial X_{ij}}\frac{\partial D'}{\partial X_{kl}}+
      \left[\frac{\partial^2 D'}{\partial X_{ij}\partial X_{kl}},D'\right]\right)[X,Z]_{ij}[X,Z]_{kl}=\\
    \Wr'+\sum_{ij}[X,Z]_{ij}\frac{\partial \Wr'}{\partial X_{ij}}+\sum_{ij,kl}\frac{\partial^2 \Wr'}{\partial X_{ij}\partial X_{kl}}[X,Z]_{ij}[X,Z]_{kl}=\Wr^{lin}.
  \end{multline*}
The last equality follows from the fact that \(\Wr'\) is linear of \(X\) and that \(\frac{\partial\Wr'}{\partial X}=[Y,Z]\mbox{ mod }\mathfrak{m}(Z)^2\)

The maps between the spaces of morphisms could be reconstructed from the above considerations.
\end{proof}

We expect there is an extension of our argument to the higher orders:
\begin{conj} For every \(k>0\) there is a natural functor
  \[\MF(\calXr_{r,n}^{\sim 2},\xWl)=\MF(\calXr_{r,n}^{\sim 2},\Wr).\]
\end{conj}

\bibliographystyle{unsrt}
\bibliography{expanded}
\end{document}